\documentclass{article}
\usepackage{amsmath,amsfonts,amsthm,amssymb,verbatim,enumerate,a4wide}

\theoremstyle{plain}
\newtheorem{thm}{Theorem}[section]
\newtheorem{prop}[thm]{Proposition}
\newtheorem{cor}[thm]{Corollary}
\newtheorem{lem}[thm]{Lemma}

\theoremstyle{definition}

\newtheorem{ex}{Example}[section]

\theoremstyle{remark}

\begin{document}
\begin{center}

{\huge Substitution-based structures with\\[0.75ex] absolutely 
continuous spectrum}\vspace{4ex}

{\Large Lax Chan, Uwe Grimm and Ian Short}\vspace{4ex}

{\small
School of Mathematics and Statistics, The Open University,
Milton Keynes MK7 6AA, UK}\vspace{2ex}
\end{center}

\begin{abstract}
By generalising Rudin's construction of an aperiodic sequence, we derive new
substitution-based structures which have purely absolutely continuous
diffraction and mixed dynamical spectrum, with absolutely continuous
and pure point parts. We discuss several examples, including a 
construction based on Fourier matrices which yields constant-length
substitutions for any length.\vspace{2ex}
\end{abstract}

\section{Introduction}\sloppy

Substitution dynamical systems are widely used as toy models for
aperiodic order in one dimension~\cite{PF02,MQ10}. By an argument of
Dworkin~\cite{SD93}, the diffraction spectrum of these systems is
related to part of the dynamical spectrum, which is the spectrum of a
unitary operator acting on a Hilbert space, as induced by the shift
action. We refer the readers to~\cite{BL16} and references therein for
recent developments and the current knowledge of the relationship between
these different spectral characterisations. Here we are interested in
systems that feature absolutely continuous spectra, in spite of being
perfectly ordered.

A paradigm of such a system is the (binary) Rudin--Shapiro or
Golay--Shapiro sequence.\footnote{For simplicity, we will refer to
  this sequence as the Rudin--Shapiro sequence, as this is the more
  commonly used term.} It was introduced in~\cite{G51,Rud59,Sha51} in
answer to a question raised by Salem~\cite{Rud59} in the context of
harmonic analysis; see also \cite[Sec.~4.7.1]{BG13}. This sequence,
represented by a Dirac comb with balanced weights ($\pm 1$), is a
substitution-based structure with purely absolutely continuous
diffraction spectrum, so it has a mixed dynamical spectrum, with a
pure point part arising from the underlying constant-length
substitution structure. Indeed, this deterministic sequence has the
stronger property that its two-point correlations vanish exactly for
\textit{any} non-zero distance; a direct proof of this property can be
found in \cite[Sec.~10.2]{BG13}; see also \cite{PF02}. Recall that the
diffraction measure is the Fourier transform of the autocorrelation
measure, which in this case is just $\delta_{0}$, so the diffraction
measure is Lebesgue measure.  Some generalisations of the
Rudin--Shapiro sequence were provided in~\cite{AL91}, but to date
relatively few examples of substitution-based sequences of this type
are known explicitly. There are good reasons for this, as one would
expect a generic substitution-based structure to produce a singular
continuous spectrum \cite{AB14}.

In~\cite{Fra03}, a systematic generalisation of the Rudin--Shapiro
system to higher-dimensional substitutions was derived. It employs
Hadamard matrices (matrices with elements $\pm 1$ whose rows are
mutually orthogonal). The underlying systems are symbolic
constant-length substitutions on a finite alphabet $\mathcal{A}$,
based on arrangements of letters on the (hyper)cubic lattice
$\mathbb{Z}^{d}$. Letters in the alphabet are paired, so for each
letter $a\in\mathcal{A}$ there is a twin letter $\overline{a}\in\mathcal{A}$, with $\overline{a}\neq a$ and $\overline{\overline{a}}=a$.
In particular, the author proved the following result, where
$\mathbb{X}$ denotes the hull of the substitution, $\mu$ the
corresponding invariant measure, $H_{\! D}$ the discrete spectrum, and
$Z(f)$ the cyclic subspace associated to a function $f\in
L^{2}(\mathbb{X},\mu)$.

\begin{thm}[\cite{Fra03}]\label{thm:Fra03}
  Let\/ $(\mathbb{X},\mathbb{Z}^{d},\mu)$ be a dynamical system
  associated to an aperiodic\/ $\mathbb{Z}^{d}$-substitution
  subject to the conditions that\/\footnote{We refer to the original
    article for more details on these conditions.}
\begin{itemize}
\item each letter in the alphabet\/ $\mathcal{A}$ is only allowed to
  appear in the position given by its underlying number\/ $($so the
  images of letters under the substitution differ only in the number
  and/or position of the bars that distinguish paired letters\/$)$;
\item paired letters are substituted by corresponding paired blocks;
\item the symbol matrix of the corresponding substitution is a
  Hadamard matrix.
\end{itemize} 
Then, there exist functions\/ $f^{}_1,\ldots, f^{}_{K}\in L^{2}(\mathbb{X},\mu)$,
each with spectral measure equal to Lebesgue measure, such
that
\[
   L^{2}(\mathbb{X},\mu) \,= \, H_{D}\oplus Z(f^{}_1)\cdots\oplus Z(f^{}_K).
\]
\end{thm}

In this article, we provide further examples of substitution-based
structures with\break Lebesgue spectrum. These systems do not satisfy the
last condition of Theorem~\ref{thm:Fra03}, although they still appear
to have a close relationship to Hadamard matrices and their complex
analogues. Our approach is based on modifying and extending the
original construction of Rudin~\cite{Rud59}. As a consequence, our
examples are sequences $(\varepsilon_{n})^{}_{n\in\mathbb{N}}$ that
satisfy the property
\begin{equation}\label{equation:1a}
   \sup_{|x|=1}^{}\bigg\lvert\sum_{n=1}^\infty \varepsilon_{n}x^n\bigg\rvert 
   \,\leq\, C\, N^{\frac{1}{2}},
\end{equation}
where the supremum is taken over complex numbers of unit modulus. We
shall refer to this property as the \emph{root-$N$ property}. This
bound on the growth of the exponential sums implies that the
corresponding diffraction spectrum is purely absolutely continuous;
compare \cite{AL91,CG17a}.

We start by revisiting the Rudin--Shapiro sequence. We then generalise
this approach in Section 3 by introducing a sequence of signs in the
recurrence relations. This results in substitution-based structures in
which the underlying substitutions are of length $2^k$ for
$k\in\mathbb{N}$.  In Section 4, we go a step further by considering
complex coefficients that are related to Fourier matrices. In this
case, we obtain new substitutions for \emph{any} length $n\ge 3$, which
give rise to weighted Dirac combs with (in the balanced weight case)
purely absolutely continuous diffraction.

\section{The Rudin--Shapiro sequence revisited}

Let us briefly review Rudin's original construction of the
Rudin--Shapiro (RS) sequence. For details of the proof, see~\cite{Rud59}.

We start by defining two sequences of polynomials, $(P_{k}(x))_{k\in\mathbb{N}_0}$ and
$(Q_{k}(x))_{k\in\mathbb{N}_0}$, where $P_k$ and $Q_k$ both have degree $2^{k}$. They
are determined by the initial choices $P^{}_{0}(x)=Q^{}_{0}(x)=x$
together with the recurrence relations
\begin{equation}\label{equation:2a}
\begin{split}
  P_{k+1}(x)\, =\, P_{k}(x)+x^{2^{k}}Q_{k}(x),\\
  Q_{k+1}(x)\, =\, P_{k}(x)-x^{2^{k}}Q_{k}(x).
\end{split}
\end{equation} 
It is clear from Eq.~\eqref{equation:2a} that the first
$2^k$ terms of $P_{k+1}(x)$ and of $Q_{k+1}(x)$ coincide with those of
$P_{k}(x)$, and that their remaining terms differ by a
sign. By construction, $P_{k}(x)$ is of the form
\begin{equation}\label{equation:2b}
  P_{k}(x)\, =\,\sum_{n=1}^{2^{k}}\,\varepsilon_{n}\,x^n,
\end{equation}
so we can define a binary sequence
$(\varepsilon_{n})^{}_{n\in\mathbb{N}}\in\{\pm 1\}^{\mathbb{N}}$ from
the corresponding coefficients. This is the binary RS
sequence. For example, for $k=3$, we have the polynomial
$P_3(x)=x+x^{2}+x^{3}-x^{4}+x^{4}(x+x^{2}-x^{3}+x^{4})$, from which we
read off the sequence $111\overline{1}11\overline{1}1$, where here (and
henceforth) we use the convention that $\overline{1}=-1$.

If $a^{}_{k}=\varepsilon^{}_{1}\varepsilon^{}_{2}\dotsb \varepsilon^{}_{2^{k}}
\in\{\pm 1\}^{2^{k}}$ denotes the word of length $2^k$ of coefficients
of $P^{}_{k}(x)$, and $b^{}_{k}$ denotes the corresponding word for
$Q^{}_{k}(x)$, then the recurrence relations of Eq.~\eqref{equation:2a}
correspond to the concatenation relations
\begin{equation}\label{eq:concat}
\begin{split}
    a^{}_{k+1} \, & = \, a^{}_{k} b^{}_{k},\\
    b^{}_{k+1} \, & = \, a^{}_{k} \overline{b^{}_{k}},
\end{split}
\end{equation}
on words in the two-letter alphabet $\{1,\overline{1}\}$, with initial
values $a_{0}=b_{0}=1$. 

The concatenation relations \eqref{eq:concat} can be seen to
correspond to the substitution rule $A\mapsto AB$, $B\mapsto
A\overline{B}$ on the four-letter alphabet
$\{A,B,\overline{A},\overline{B}\}$, which upon completion to a
four-letter substitution rule becomes
\begin{equation}\label{eq:S+rule}
    S_{+}^{}\!:\quad A\mapsto AB, \quad B\mapsto A\overline{B}, \quad
   \overline{A}\mapsto \overline{A}\overline{B},\quad
    \overline{B}\mapsto \overline{A}B,
\end{equation}
so that the `bar' operation is compatible with the substitution; see
\cite{BG16} for more on substitutions that feature a `bar-swap
symmetry' of this kind. This substitution is often referred to as the
four-letter RS substitution rule. Clearly, by induction,
this rule gives rise to the concatenation relations
\begin{equation}\label{eq:concat4}
   \begin{split}
    A^{}_{k+1} \, & = \, A^{}_{k} B^{}_{k},\\
    B^{}_{k+1} \, & = \, A^{}_{k} \overline{B^{}_{k}},
\end{split} 
\end{equation}
which have the same structure as Eq.~\eqref{eq:concat}, but work on
the four-letter alphabet $\{A,B,\overline{A},\overline{B}\}$ instead of the two-letter alphabet
$\{1,\overline{1}\}$.  The connection between the two is provided by
the map
\begin{equation}\label{eq:phi} 
     \varphi\!:\; \begin{cases}
       A,B\mapsto 1,\\[0.5ex]
      \overline{A},\overline{B}\mapsto\overline{1}.\end{cases}
\end{equation}

Iterating the substitution rule $S_{+}^{}$ (defined by Eq.~\eqref{eq:S+rule}) on the initial letter $A$
(which corresponds to our case) gives
\[
   A\mapsto AB\mapsto ABA\overline{B}\mapsto 
   ABA\overline{B}AB\overline{A}B\mapsto 
   ABA\overline{B}AB\overline{A}B
   ABA\overline{B}\overline{AB}A\overline{B}\mapsto 
\dotsb \longrightarrow w_{+}^{},
\]
which converges (in the local topology) to an infinite fixed point
word $w_{+}$.\footnote{Here and below we use the initial
  letter $A$ to construct a fixed point sequence $w$. There will
  always be a second fixed point sequence, which due to the bar-swap
  symmetry of our substitutions is just $\overline{w}$, which can be
  obtained by iterating $S_+$ on the initial letter $\overline{A}$.}  We
denote the corresponding hull, which is the closure of the orbit under
the shift, by $\mathbb{X}^{}_{+}$. The binary RS sequence
is then recovered as the image under the factor map $\varphi$ of
Eq.~\eqref{eq:phi}, which reproduces the sequence
$(\varepsilon_{n})^{}_{n\in\mathbb{N}}\in\{\pm 1\}^{\mathbb{N}}$. Note
that there is no two-letter substitution rule for this sequence,
unless you work with a staggered substitution with different rules
for even and odd positions along the word; see
\cite[Sec.~4.7.1]{BG13}. 

The main ingredient in Rudin's proof \cite{Rud59} of the root-$N$
property \eqref{equation:1a} for the binary sequence is the
parallelogram law,
\begin{equation}\label{equation:2c}
   |\alpha+\beta|^{2}+|\alpha-\beta|^{2}
   \, =\, 2|\alpha|^{2}+2|\beta|^{2},
\end{equation}
where $\alpha,\beta\in\mathbb{C}$, and this will also be the case in
our generalisations discussed below.  Note that this means that the
consequences on the spectral properties specifically apply to the
binary sequence $(\varepsilon_{n})^{}_{n\in\mathbb{N}}\in\{\pm
1\}^\mathbb{N}$, and that this argument does not directly provide
information about the spectral properties of the underlying
four-letter sequence obtained from the substitution rule of
Eq.~\eqref{eq:S+rule}.

\section{Modifying Rudin's construction}

Let us now introduce some modifications to the original construction
of Rudin, and show that our newly derived recurrence relations still
satisfy the root-$N$ property of Eq.~\eqref{equation:1a}. Following
this, we compute some concrete examples and derive the corresponding
substitution systems, in the same way as for the RS
sequence above.

We again work with two sequences of polynomials
$(P^{}_{k}(x))_{k\in\mathbb{N}_0}$ and
$(Q^{}_{k}(x))_{k\in\mathbb{N}_0}$, with
$P^{}_{0}(x)=Q^{}_{0}(x)=x$. By introducing additional signs
$\sigma^{}_{k}\in\{\pm 1\}$ in the recurrence relations of
Eq.~\eqref{equation:2a}, we consider
\begin{equation}\label{equation:3a}
\begin{split}
P^{}_{k+1}(x)\, &=\, P^{}_{k}(x)+\sigma_{k}^{}\,x^{2^{k}}Q^{}_{k}(x),\\
Q^{}_{k+1}(x)\, &=\, P^{}_{k}(x)-\sigma_{k}^{}\,x^{2^{k}}Q^{}_{k}(x),
\end{split}
\end{equation}
for $k\in\mathbb{N}_{0}$. At this stage, we do not yet specify the
values of $\sigma_{k}$. 

Clearly, the RS case corresponds to the choice
$\sigma_{k}=1$ for all $k\in\mathbb{N}_{0}$. If instead one chooses
$\sigma_{k}=-1$ for all $k\in\mathbb{N}_{0}$, the recurrence relations
correspond to the substitution
\begin{equation}\label{eq:S-rule}
  S_{-}^{}\!:\quad A\mapsto A\overline{B}, \quad B\mapsto AB, \quad
  \overline{A}\mapsto \overline{A}B,\quad \overline{B}\mapsto
  \overline{AB}.
\end{equation}
Its one-sided fixed point $w^{}_{-}$, obtained by iterating $S_{-}^{}$ on
the letter $A$, 
\[
A \mapsto  A\overline{B} \mapsto A\overline{BAB}
\mapsto A\overline{BABA}B\overline{AB} \mapsto
A\overline{BABA}B\overline{ABA}BAB\overline{A}B\overline{AB}
\mapsto \dotsb \longrightarrow w^{}_{-},
\]
gives rise to the hull $\mathbb{X}^{}_{-}$. It is easy to verify that
$\mathbb{X}^{}_{+}\ne\mathbb{X}^{}_{-}$, since there are subwords of
length six in $w^{}_{+}$ (such as $BABA\overline{BA}$ or
$BAB\overline{A}BA$) which do not occur as subwords of $w^{}_{-}$, and
vice versa. Indeed, the same holds true for the corresponding binary
sequences and their hulls
$\mathbb{Y}^{}_{+}:=\varphi(\mathbb{X}^{}_{+})$ and $\mathbb{Y}^{}_{-}:=\varphi(\mathbb{X}^{}_{-})$. For example,
$\varphi(BABA\overline{BA})=1111\overline{11}$ is a subword of
$\varphi(w^{}_{+})$ but not of $\varphi(w^{}_{-})$, as $ABAB$ does not
appear in either $w^{}_{+}$ or $w^{}_{-}$. Observe that, in fact, $\varphi$ induces
a bijection between the hulls, so $\mathbb{X}^{}_{+}$ and
$\mathbb{Y}^{}_{+}$ (and also $\mathbb{X}^{}_{-}$ and
$\mathbb{Y}^{}_{-}$) are mutually locally derivable, and the
corresponding four-letter and two-letter dynamical systems (under the
shift action) are topologically conjugate; compare
\cite[Rem.~4.11]{BG13}. This can, for instance, be seen by realising
that the subword $1111$, which occurs in both $\varphi(w^{}_{+})$ and
$\varphi(w^{}_{-})$ with bounded gaps, has the unique preimage $BABA$
in both $w^{}_{+}$ and $w^{}_{-}$.  

One can also verify that the substitution matrices for $S^{}_{+}$ and
$S^{}_{-}$ have different eigenvalues ($2$, $\pm\sqrt{2}$ and $0$ for
$S^{}_{+}$ and $2$, $1\pm \mathrm{i}$ and $0$ for $S^{}_{-}$), so the
corresponding substitution dynamical systems cannot be conjugate (as
they have different dynamical zeta functions). However, this does not
answer the question whether $\mathbb{X}^{}_{+}$ and
$\mathbb{X}^{}_{-}$ are mutually locally derivable or not, because
this difference vanishes if you look at the eighth power of the
substitutions.  

\begin{prop}\label{prop:rootN}
  The sequence of coefficients\/
  $(\varepsilon_{n})^{}_{n\in\mathbb{N}}$ of the functions\/ $P_{k}$,
  $k\in\mathbb{N}_{0}$, defined by the recurrence relations of
  Eq.~\eqref{equation:3a}, satisfy the root-$N$ property of
  Eq.~\eqref{equation:1a}.
\end{prop}
\begin{proof}
The proof proceeds by induction. Consider the case $|x|=1$. 
By the recurrence relations \eqref{equation:3a}, we then have 
\[
 |P^{}_{k+1}(x)|^2+|Q^{}_{k+1}(x)|^{2} \,=\, |P^{}_{k}(x)+
\sigma_{k}^{}\,x^{2^{k}}Q^{}_{k}(x)|^{2}+|P^{}_{k}(x)-
\sigma_{k}^{}\,x^{2^{k}}Q^{}_{k}(x)|^{2}.
\]
Applying the parallelogram law \eqref{equation:2c}, we find that
\[
 |P^{}_{k+1}(x)|^2+|Q^{}_{k+1}(x)|^{2}\, =\, 
  2\bigl( |P^{}_{k}(x)|^{2}+ |Q^{}_{k}(x)|^{2}\bigr).
\]
Since $|P_{0}(x)|^{2}+|Q_{0}(x)|^{2}=2$, we can conclude that
\[
  |P^{}_{k}(x)|^{2}+|Q^{}_{k}(x)|^{2}\, =\, 2^{k+1},
\]
and hence
\begin{equation}\label{equation:3b}
  |P_{k}(x)|\, \leq\, 2^{\frac{k+1}{2}}.
\end{equation} 
This proves the root-$N$ property for $N=2^{k}$.

In order to tackle the case when $N$ is not necessarily a power of $2$, we 
define partial sums of $P^{}_{k}$ and $Q^{}_{k}$ as follows,
\begin{equation}\label{equation:3c}
  P^{}_{k|m}(x)^{}\, =\,\sum_{n=1}^{m}\varepsilon^{}_{n}x^{n},\quad
  Q^{}_{k|m}(x)^{}\, =\, \sum_{n=1}^{m}\gamma^{}_{n}x^{n},
\end{equation}
where $2^{k-1}< m\leq 2^{k}$, $k\in\mathbb{N}_0$, and where $\varepsilon_{n},\gamma_{n}\in\{\pm 1\}$
are the corresponding coefficients. We now show that these satisfy
 \begin{equation}
 \label{equation:3d}
        \left|P^{}_{k|m}(x)\right| \, \leq\, G\,2^{\frac{k}{2}}
\quad \text{and}\quad
        \left|Q^{}_{k|m}(x)\right| \,\leq\, G\,2^{\frac{k}{2}} 
\end{equation} 
for all $|x|=1$ and $k\in\mathbb{N}_{0}$, where $G=2+2^{\frac{1}{2}}$.

The above estimates are obviously true for $k=0$. Suppose that they
hold for some $k\in\mathbb{N}_0$, and consider an integer $m$ with
$2^{k}< m\leq 2^{k+1}$. By using the triangle inequality together with
Eqs.~\eqref{equation:3b} and~\eqref{equation:3d}, we obtain
 \begin{equation}
 \label{equation:3e}
 \bigl|P^{}_{k+1|m}(x)\bigr| \,\leq\, 
  \bigl|P^{}_{k}(x)\bigr|+\bigl|Q^{}_{k|m-2^{k}}(x)\bigr|
 \,\leq\, 2^{\frac{k+1}{2}}+G\,2^{\frac{k}{2}}= \, G\, 2^{\frac{k+1}{2}},
 \end{equation}
which establishes Eq.~\eqref{equation:3d} for $k+1$. The same argument
clearly works for $Q^{}_{k+1|m}(x)$.
 
To complete the proof, suppose that $2^{k-1} < N\leq 2^{k}$. By
Eq.~\eqref{equation:3d}, we have
 \[
  \bigl|P^{}_{k|N}(x)\bigr|\,\leq\, 
  (2+2^{\frac{1}{2}})2^{\frac{k}{2}}\,\leq\,
  2(1+2^{\frac{1}{2}})N^{\frac{1}{2}},
\]
which shows that the root-$N$ property holds.
\end{proof}

\begin{cor}
  Whatever the choice of the signs\/ $\sigma_{k}\in\{\pm 1\}$ in
  Eq.~\eqref{equation:3a}, the corresponding sequence\/
  $(\varepsilon_{n})^{}_{n\in\mathbb{N}}$ is balanced.
\end{cor}
\begin{proof}
The average value of the first $N$ coefficents is given by
\[
 \frac{1}{N} \sum_{n=1}^{N} \varepsilon_{n} \, = \, \frac{1}{N} P^{}_{k|N}(1)
\]
for sufficiently large $k$. By Proposition~\ref{prop:rootN}, this is
bounded by
\[
    \left| \frac{1}{N} \sum_{n=1}^{N} \varepsilon_{n} \right| \, \le \, 
    2 (1+2^{\frac{1}{2}}) N^{-\frac{1}{2}}
\]
which goes to $0$ as $N\to\infty$.
\end{proof}

As mentioned previously, the root-$N$ property implies absolute
continuity of the diffraction for the binary sequence.

\begin{cor}
For any series of coefficients\/  $(\varepsilon_{n})^{}_{n\in\mathbb{N}}$ 
as in Proposition~\emph{\ref{prop:rootN}}, the corresponding
Dirac comb\/ $\sum_{n\in\mathbb{N}}\varepsilon_{n}\delta_{n}$ has
purely absolutely continuous diffraction.\qed
\end{cor}
 

We now consider some examples.

\begin{ex}\label{ex:S-+}
  Let us start with the choice $\sigma_{k}^{}=(-1)^{k+1}$, so the
  signs in the recurrence relations for the polynomials alternate, and
  we have
\begin{equation}\label{equation:3g}
\begin{split}
  P^{}_{k+1}(x)\, =\, P^{}_{k}(x)+(-1)^{k+1}x^{2^{k}}Q^{}_{k}(x),\\
  Q^{}_{k+1}(x)\, =\, P^{}_{k}(x)-(-1)^{k+1}x^{2^{k}}Q^{}_{k}(x),
\end{split}
\end{equation}
for $k\in\mathbb{N}_{0}$. We could now read off the corresponding substitution rule just as we did for the RS substitution, but this case is more complicated because of the alternating signs. One way to overcome this problem is
to look at two consecutive steps at once,
\begin{equation}\label{equation:3h}
 \begin{split}
  P^{}_{k+2}(x)\, =\, P^{}_{k}(x) + (-1)^{k+1}x^{2^{k}}Q^{}_{k}(x) + 
                 (-1)^{k+2}x^{2\cdot 2^{k}} P^{}_{k}(x)
                 + x^{3\cdot 2^{k}}Q^{}_{k}(x),\\
  Q^{}_{k+2}(x)\, =\, P^{}_{k}(x) + (-1)^{k+1}x^{2^{k}}Q^{}_{k}(x) - 
                 (-1)^{k+2}x^{2\cdot 2^{k}} P^{}_{k}(x)
                 - x^{3\cdot 2^{k}}Q^{}_{k}(x).
\end{split}
\end{equation}
Choosing $k$ to be even (which corresponds to the case we are
interested in, since our recursion starts with $k=0$) and associating
 letters $A$ and $B$, and their counterparts
$\overline{A}$ and $\overline{B}$, to the sequences corresponding to $P$
and $Q$, we obtain the substitution rule
\begin{equation}
\label{equation:3i}
  S_{-+}^{}\!:\quad 
  A\mapsto A\overline{B}AB,\quad 
  B\mapsto A\overline{B}\overline{A}\overline{B},\quad 
  \overline{A}\mapsto \overline{A}B\overline{A}\overline{B},\quad
  \overline{B} \mapsto \overline{A}BAB.
\end{equation} 
This is a substitution of constant length four, because we used a
double step of the recursion, and Eq.~\eqref{equation:3h} corresponds
to concatenation of four sets of coefficients. As before, a one-sided
fixed point sequence $w^{}_{-+}$ is obtained from iterating the
substitution on the initial letter $A$,
\[
  A\mapsto A\overline{B}AB\mapsto A\overline{B}AB \overline{A}BAB
  A\overline{B}AB A\overline{B}\overline{A}\overline{B}\mapsto\cdots \longrightarrow w_{-+}.
\]
By  mapping $A$, $B$ to $1$ and $\overline{A}$, $\overline{B}$ to
$\overline{1}=-1$ using the map $\varphi$ of Eq.~\eqref{eq:phi}, we
obtain the binary sequence
$v^{}_{-+}=\varphi(w^{}_{-+})=
1\overline{1}11\overline{1}1111\overline{1}111\overline{111}\ldots$
as our new RS-type sequence.

Alternatively, one can see the substitution $S_{-+}^{}$ as the
composition of the two substitutions $S_{+}$ and $S_{-}$ from
Eqs.~\eqref{eq:S+rule} and \eqref{eq:S-rule}, in the sense that
$S_{-+}^{}=S_{-}^{}\circ S_{+}^{}$.  To see this explicitly, let us
verify the composition on the letters $A$ and $B$,
\begin{align*}
&A\xrightarrow{S_{+}^{}}AB\xrightarrow{S_{-}^{}}A\overline{B}AB,\\
&B\xrightarrow{S_{+}^{}}A\overline{B}\xrightarrow{S_{-}^{}}
 A\overline{B}\overline{A}\overline{B},
\end{align*}
with the corresponding result for $\overline{A}$ and $\overline{B}$
following by the bar-swap symmetry.

From Proposition~\ref{prop:rootN}, we conclude that the binary
sequence $v^{}_{-+}=\varphi(w^{}_{-+})$  satisfies the root-$N$ property, and hence 
the corresponding diffraction measure is absolutely continuous.
\end{ex}

Our next example is closely related. We again alternate the signs in
the recursion, but shifted by one. Maybe surprisingly, this 
produces a different sequence of coefficients.

\begin{ex}\label{ex:S+-}
  Here we choose $\sigma_{k}=(-1)^{k}$. The
  recurrence relations are now
\begin{equation}\label{equation:3j}
\begin{split}
P^{}_{k+1}(x)\, =\, P^{}_{k}(x)+(-1)^{k}x^{2^{k}}Q^{}_{k}(x),\\
Q^{}_{k+1}(x)\, =\, P^{}_{k}(x)-(-1)^{k}x^{2^{k}}Q^{}_{k}(x),
\end{split}
\end{equation}
for $k\in\mathbb{N}_{0}$. Using the same approach as above, this 
gives rise to the substitution rule
\begin{equation}
\label{equation:3l}
  S_{+-}^{}\!:\quad 
  A\mapsto AB\overline{A}B,\quad 
  B\mapsto ABA\overline{B},\quad 
  \overline{A}\mapsto \overline{A}\overline{B}A\overline{B},\quad
  \overline{B}\mapsto \overline{A}\overline{B}\overline{A}B.
\end{equation} 
This rule can also be expressed as 
the composition of the two substitution systems $S_{+}$ and $S_{-}$, 
this time as $S_{+-}=S_{+}^{}\circ S_{-}^{}$, because
\begin{align*}
  &A\xrightarrow{S_{-}^{}}A\overline{B}\xrightarrow{S_{+}^{}}AB\overline{A}B,\\
  &B\xrightarrow{S_{-}^{}}AB\xrightarrow{S_{+}^{}}ABA\overline{B},
\end{align*}
and the relations for the barred letters follow by bar-swap
symmetry. Again, Proposition~\ref{prop:rootN} shows that the
corresponding binary sequence $v^{}_{+-}=\varphi(w^{}_{+-})$, where
$w^{}_{+-}$ denotes the fixed point of $S_{+-}$ obtained by iterating
$S_{+-}$ on the letter $A$, satisfies the root-$N$ property and hence
gives rise to a Dirac comb with absolutely continuous diffraction
measure.
\end{ex}

The observations of Examples~\ref{ex:S-+} and \ref{ex:S+-} are in line
with the fact that we use the recurrence relations for the two
different signs, corresponding to $S^{}_{+}$ and $S^{}_{-}$,
alternatingly here, so the net substitution is the composition of
both. Clearly, the two substitutions $S^{}_{-+}$ and $S^{}_{+-}$ and
their respective hulls are closely related.

\begin{lem}
  The hulls\/ $\mathbb{X}_{-+}$ and\/ $\mathbb{X}^{}_{+-}$ of the
  substitutions\/ $S^{}_{-+}$ and\/ $S^{}_{+-}$ defined by
  Eqs.~\eqref{equation:3i} and \eqref{equation:3l} satisfy the
  relations
\begin{align*}
\mathbb{X}^{}_{+-}&= S^{}_{+}(\mathbb{X}^{}_{-+}) \cup 
                    TS^{}_{+}(\mathbb{X}^{}_{-+}),\\
\mathbb{X}^{}_{-+}&= S^{}_{-}(\mathbb{X}^{}_{+-}) \cup
                    TS^{}_{-}(\mathbb{X}^{}_{+-}),
\end{align*}
where\/ $T$ denotes the shift map on
$\{A,\overline{A},B,\overline{B}\}^\mathbb{N}_{}$. 
\end{lem}
\begin{proof}
  If $w^{}_{-+}$ and $w^{}_{+-}$ denote fixed point sequences for
  $S^{}_{-+}$ and $S^{}_{+-}$, then the (one-sided) hulls $\mathbb{X}_{-+}$ and
  $\mathbb{X}^{}_{+-}$ are given as the orbit closures of the fixed
  points under the shift map $T$. Now, $S^{}_{-+}=S^{}_{-}\circ
  S^{}_{+}$ implies that
\[
    S^{}_{-+}(S_{-}w^{}_{+-}) \,= \,
   \bigl(S^{}_{-}\circ S^{}_{+}\circ S^{}_{-}\bigr)\, w^{}_{+-}
    \,=\, S^{}_{-}(S^{}_{+-}w^{}_{+-})\, =\, S_{-}w^{}_{+-},
\] 
which shows that $S_{-}w^{}_{+-}$ is a fixed point of $S^{}_{-+}$. Similarly,
since $S^{}_{+-}=S^{}_{+}\circ S^{}_{-}$, we have 
\[
   S^{}_{+-}(S_{+}w^{}_{-+}) \,= \,
   \bigl( S^{}_{+}\circ S^{}_{-}\circ S^{}_{+}\bigr)\, w^{}_{-+}
    \,=\, S^{}_{+}(S^{}_{-+}w^{}_{-+}) \,=\, S_{+}w^{}_{-+},
\]
and consequently $S_{+}w^{}_{-+}$ is a fixed point of $S^{}_{+-}$.  

Since the substitutions $S^{}_{+}$ and $S^{}_{-}$ have constant length two, we have
\[
   S^{}_{+}\circ T\, =\, T^{2} \circ S^{}_{+}\quad\text{and}\quad    S^{}_{-}\circ T\, =\, T^{2} \circ S^{}_{-},
\]
which implies that $S^{}_{+}(\mathbb{X}^{}_{-+})$ is the subset of
$\mathbb{X}^{}_{+-}$ of all sequences starting with a letter $A$ or
$\overline{A}$, since only even shifts are included. By continuity of
the action, limits are included, so the closure does not add any
additional elements. Hence the union $S^{}_{+}(\mathbb{X}^{}_{-+})
\cup TS^{}_{+}(\mathbb{X}^{}_{-+})$ gives the complete hull
$\mathbb{X}^{}_{+-}$, and the analogous result holds for the case
where the signs are interchanged.
\end{proof}

Notice that, despite this close connection, the two hulls
$\mathbb{X}^{}_{-+}$ and $\mathbb{X}^{}_{+-}$ are indeed different, as
can be verified by considering words of length six. Note also that the
eigenvalues of the substitution matrices of $S^{}_{-+}$ and
$S^{}_{+-}$ are again different; they are $4$, $2$ (twice) and $0$ for
$S^{}_{-+}$, and $4$, $\pm 2$ and $0$ for $S^{}_{+-}$. The question of
whether the two hulls are mutually locally derivable remains open.

Still, the following result shows that the two systems are intimately
linked.

\begin{prop}
  $(\mathbb{X}_{-+}, T)$ is conjugate to the induced system of\/
  $(\mathbb{X}_{+-}, T)$ on the subset $S_{+}(\mathbb{X}_{-+})$,
  and\/ $(\mathbb{X}_{+-}, T)$ is conjugate to the induced system
  of\/ $(\mathbb{X}_{-+}, T)$ on the subset\/
  $S_{-}(\mathbb{X}_{+-})$.
\end{prop}

\begin{proof}
  Here, we prove the first claim; the second follows analogously.  As
  mentioned above, $S_{+}(\mathbb{X}^{}_{-+}) = [A]\cup
  [\,\overline{A}\, ] \subset \mathbb{X}^{}_{+-}$, where the brackets
  denote cylinder sets of words starting with this letter.  Now,
  consider the return time function~\cite[Sec.~2.2]{Dur00}, that is,
  the return time of the fixed point generated by $S_{+-}$ to the
  clopen set $[A]\cup[\,\overline{A}\, ]$,
\[r_{[A]\cup[\,\overline{A}\,]}^{}=\inf\{n>0:
  T^{n}(w_{+-})\in [A]\cup [\,\overline{A}\, ]\}.
\] 
As $S_{+}$ is a substitution of length two, each letter is mapped into
a length two word starting with $A$ or $\overline{A}$ and it follows
$r_{[A]\cup[\,\overline{A}\,]}^{}=2$. The induced map is then given by
$T^{2}$, which maps the set $[A]\cup[\,\overline{A}\, ]$ onto
itself. Hence, $([A]\cup[\, \overline{A}\, ], T^{2})$ is the induced
system. As $S_{+}$ is an injective map from $\mathbb{X}_{-+}$ to
$S_{+}(\mathbb{X}_{-+})$, the claimed conjugacy follows
\cite[Sec.~2.1]{DOP15}.
\end{proof}

By using the following result, the induced systems inherit the
spectral properties of the conjugated systems.

\begin{thm}[{\cite[Thm.~2.9]{Wal82}}]
  Let\/ $T_{i}$ with\/ $i\in\{1,2\}$ be measure-preserving
  transformations of probability spaces. If\/ $T_1$ and\/ $T_2$ are
  conjugate, then they are spectrally isomorphic.
\end{thm}

Note that, as previously, the four-letter hull $\mathbb{X}^{}_{-+}$
and two-letter hull $\varphi(\mathbb{X}^{}_{-+})$ are mutually locally
derivable (as are $\mathbb{X}^{}_{+-}$ and
$\varphi(\mathbb{X}^{}_{+-})$), and the corresponding dynamical
systems are hence topologically conjugate. The argument is the same as
above; the subword $1111$, which occurs in both $\varphi(w^{}_{-+})$
and $\varphi(w^{}_{+-})$ with bounded gaps, has the unique preimage
$BABA$ in both $w^{}_{-+}$ and $w^{}_{+-}$.

The observations of Examples~\ref{ex:S-+} and \ref{ex:S+-} suggest the
following general picture.

\begin{prop}\label{prop:gen}
  Let\/ $(\sigma^{}_{k})^{}_{k\in\mathbb{N}}\in\{\pm 1\}^{\mathbb{N}_{0}}$
  be a given sequence. Then, for any\/ $k\in\mathbb{N}_{0}$, the
  sequence of coefficients of the polynomial\/ $P_{k}$ defined by
  Eq.~\eqref{equation:3a} is the image under the map\/ $\varphi$ of
  $A^{}_{k}=S^{}_{\sigma^{}_{0}}\circ S^{}_{\sigma^{}_{1}}\circ\dots\circ 
         S^{}_{\sigma^{}_{k-1}}A$.
\end{prop}
\begin{proof}
Let $a^{}_{k}:=\varepsilon^{}_{1}\varepsilon^{}_{2}\dots\varepsilon^{}_{2^{k}}
\in\{\pm 1\}^{2^{k}}$ denote the word of length $2^k$ of coefficients
of $P^{}_{k}(x)$, and $b^{}_{k}$ the corresponding word for
$Q^{}_{k}(x)$.
Then, the recurrence relations of Eq.~\eqref{equation:3a}
correspond to the concatenation relations
\begin{equation}\label{eq:concat2}
    a^{}_{k+1} \,  = \, \begin{cases}  a^{}_{k}  b^{}_{k}, & \sigma^{}_{k}=1, \\
                        a^{}_{k} \overline{b^{}_{k}}, & \sigma^{}_{k}=-1,
                       \end{cases}\qquad
    b^{}_{k+1} \,  = \,  \begin{cases}  a^{}_{k}  \overline{b^{}_{k}}, & 
                       \sigma^{}_{k}=1, \\
                        a^{}_{k} b^{}_{k}, & \sigma^{}_{k}=-1,
                       \end{cases}
\end{equation}
with initial values $a_{0}=b_{0}=1$. These recurrence relations
correspond to the substitution rule $S^{}_{\sigma^{}_{k}}$, and by
induction we obtain $a^{}_{k}=\varphi(A^{}_{k})$ with
\[
    A^{}_{k} \, = \, S^{}_{\sigma^{}_{0}}\circ \dots\circ S^{}_{\sigma^{}_{k-1}} A
\]  
for any $k\in\mathbb{N}_{0}$.  
\end{proof}

Clearly, if we choose $\sigma^{}_{k}=1$ for all $k\in\mathbb{N}_{0}$,
we are back at the RS case with substitution
$S^{}_{+}$. More generally, for any periodic sequence we have the
following result.

\begin{cor}\label{cor:gen}
  Let\/ $(\sigma^{}_{k})^{}_{k\in\mathbb{N}_{0}}\in\{\pm 1\}^{\mathbb{N}_{0}}$
  be a periodic sequence of period\/ $p$, so\/
  $\sigma^{}_{k+p}=\sigma^{}_{k}$ for all $k\in\mathbb{N}_{0}$. Then, the
  sequence of coefficients of the polynomials\/ $P_{k}$ defined by
  Eq.~\eqref{equation:3a} is the image under the map\/ $\varphi$ of
  the fixed point of the substitution
\[
  S^{}_{\sigma^{}_{0}\sigma^{}_{1}\dots\sigma^{}_{p-1}} \, := \, 
  S^{}_{\sigma^{}_{0}}\circ S^{}_{\sigma^{}_{1}}\circ\dots\circ S^{}_{\sigma^{}_{p-1}}
\]
with initial letter $A$.
\end{cor}
\begin{proof}
As the sequence of signs $\sigma^{}_{k}$ is
periodic with period $p$, Proposition~\ref{prop:gen} implies that
\[
    A^{}_{np} \, = \, (S^{}_{\sigma^{}_{0}}\circ \dots\circ 
   S^{}_{\sigma^{}_{p-1}})^n A 
\]
holds for $n\in\mathbb{N}$, and the assertion follows.
\end{proof}

If the sequence is not periodic, we are in a situation that resembles
the case of random substitutions considered in \cite{MM14}. But even
then, we still have convergence in the local topology to a
well-defined binary and quaternary sequence. However, the latter is no
longer a fixed point of a primitive substitution of finite length, so
we do not know much about the corresponding hull.  Nevertheless, the
root-$N$ property and hence absolute continuity of the spectrum also
hold in this case.


\begin{ex}
Let us consider one more example, with
\[
  \sigma_{k}^{} = \begin{cases} 1, &\mbox{if } k \equiv 0, 1 \bmod 3,\\ 
  -1,& \mbox{if } k \equiv 2 \bmod 3.\end{cases}\
\] 
From Proposition~\ref{prop:gen}, we know that the corresponding
substitution is $S^{}_{++-}=S^{}_{+}\circ S^{}_{+}\circ S^{}_{-}$,
which turns out to be
\begin{align*}
&A\xrightarrow{S_{-}^{}}A\overline{B}
\xrightarrow{S_{+}^{}}AB\overline{A}B\xrightarrow{S_{+}^{}}
ABA\overline{BAB}A\overline{B}\\
&B\xrightarrow{S_{-}^{}}AB\xrightarrow{S_{+}^{}}
ABA\overline{B}\xrightarrow{S_{+}^{}}ABA\overline{B}AB\overline{A}B
\end{align*}
together with the corresponding relations for the barred letters.
\end{ex}

\begin{prop}
  Let\/ $(\sigma^{}_{k})^{}_{k\in\mathbb{N}} \in \{\pm
  1\}^{\mathbb{N}}$ be a periodic sequence of period\/ $p$ and\/
  $S^{}_{\sigma^{}_{1}\sigma^{}_{2}\dots\sigma^{}_{p}}$ be the
  corresponding substitution according to Corollary~\ref{cor:gen}.
  Its hull $\mathbb{X}^{}_{\sigma^{}_{1}\sigma^{}_{2} \dots
    \sigma^{}_{p}}$ is then mutually locally derivable with
  $\varphi(\mathbb{X}^{}_{\sigma^{}_{1}\sigma^{}_{2} \dots
    \sigma^{}_{p}})$.
\end{prop}
\begin{proof}
  Local derivability of the two-letter sequence from the four-letter
  sequence is clear, as $\varphi$ acts locally.

  To show local derivability of the four-letter sequence, note that
  $BABA$ is a legal four-letter word for $S^{}_{+}$ and $S^{}_{-}$ as
  well as for $S^{}_{-+}=S^{}_{-}\circ S^{}_{+}$ and
  $S^{}_{+-}=S^{}_{+}\circ S^{}_{-}$. Hence, it is also legal for
  $S^{}_{\sigma^{}_{1}\sigma^{}_{2}\dots\sigma^{}_{p}}=S^{}_{\sigma^{}_{1}}\circ
  S^{}_{\sigma^{}_{2}}\circ\dots\circ S^{}_{\sigma^{}_{p}}$, and
  occurs with bounded gaps in any element of the hull
  $\mathbb{X}^{}_{\sigma^{}_{1}\sigma^{}_{2}\dots\sigma^{}_{p}}$ by
  repetitivity of the hull. Since $\varphi(BABA)=1111$, the latter
  also occurs with bounded gaps in any element of the two-letter hull.

  Now, note that $ABAB$ is not a legal word for $S^{}_{+}$ or $S^{}_{-}$ (as
  its pre-image would have to be $AA$ or $BB$), or for
  $S^{}_{-+}=S^{}_{-}\circ S^{}_{+}$ or $S^{}_{+-}=S^{}_{+}\circ
  S^{}_{-}$. As a consequence, it cannot occur as a legal word for
  $S^{}_{\sigma^{}_{1}\sigma^{}_{2}\dots\sigma^{}_{p}}=S^{}_{\sigma^{}_{1}}\circ
  S^{}_{\sigma^{}_{2}}\circ\dots\circ S^{}_{\sigma^{}_{p}}$ either.

  Hence $BABA$ is the unique pre-image of $1111$, and local
  derivability follows.
\end{proof}

\section{Generalizing Rudin's argument to Fourier matrices}

We are now going to generalise Rudin's argument further by considering
\emph{complex} coefficients in our polynomials, which will naturally
lead us to look at Fourier matrices. Here, the Fourier matrix of order
$n$ is the unitary $n\!\times\! n$ matrix with elements
$\frac{1}{\sqrt{n}}\exp\bigl(\frac{2\pi
  \mathrm{i}(j-1)(k-1)}{n}\bigr)$, where $1\le j,k\le n$. The matrices
that are going to enter below will be $x$-dependent generalisations of
these Fourier matrices, without the normalisation factor $1/\sqrt{n}$.

It will be convenient to express the recurrence
relations \eqref{equation:2a} in terms of matrices as follows,
\[
\begin{pmatrix}
  P^{}_{k+1}(x)\\
  Q^{}_{k+1}(x)
 \end{pmatrix} \, =\,  
 \begin{pmatrix}
  1 &   x^{2^k} \\
 1 &  -x^{2^{k}}
 \end{pmatrix}\begin{pmatrix}
   P^{}_{k}(x)\\
   Q^{}_{k}(x)\end{pmatrix}
\, =\,  
 A^{(2,k)}
  \begin{pmatrix}
   P^{}_{k}(x)\\
   Q^{}_{k}(x)\end{pmatrix}.
\] 
Now, for $n>2$, consider a 
vector of $n$ polynomials
\[ 
 v^{}_{k}\, =\, \begin{pmatrix}
  P_{k}^{(1)}(x)\\
  \vdots\\
  P_{k}^{(n)}(x)
 \end{pmatrix}
\] 
satisfying the recurrence relation
\begin{equation}\label{equation:4a}
 v^{}_{k+1}\, =\, A^{(n,k)} v^{}_{k},
 \end{equation} 
 with initial condition $v^{}_{0}=(x,\ldots,x)^t$.  Here, $A^{(n,k)}$ is
 the $n\!\times\! n$ matrix
\begin{equation}\label{eq:genmat}
A^{(n,k)} \, =\, \begin{pmatrix}
   1 & x^{n^k} & \cdots & x^{(n-1)n^{k}} \\
   1 & \omega x^{n^k} & \cdots & \omega^{n-1} x^{(n-1)n^{k}} \\
   \vdots  & \vdots & \ddots & \vdots  \\
   1 & \omega^{n-1}x^{n^k} & \cdots & \omega^{(n-1)^2} x^{(n-1)n^{k}}
 \end{pmatrix},
\end{equation}
where $\omega=\exp(2\pi \mathrm{i}/n)$. For $x=1$, $A^{(n,k)}$ reduces
to the $n\!\times\! n$ Fourier matrix, apart from the normalisation
factor $1/\sqrt{n}$.  As a consequence, for $|x|=1$, the matrix
satisfies $(A^{(n,k)})^{\dagger}A^{(n,k)}=n\, \boldsymbol{1}^{(n)}$,
where $\boldsymbol{1}^{(n)}$ denotes the $n\!\times\! n$ unit matrix
and $M^{\dagger}=\overline{M}^{t}$ denotes the Hermitian adjoint of
the matrix (or vector) $M$.

Generalising Eq.~\eqref{equation:2b}, we can now define a sequence
$\bigl(\varepsilon^{}_{m}\bigr)_{m\in\mathbb{N}}$ of complex
coefficients $\varepsilon^{}_{m}\in\bigl\{\omega^{j} \mid 0\le j \le
n-1\bigr\}$ by
\begin{equation}\label{eq:Pcomplex}
    P_{k}^{(1)}(x) \, = \, \sum_{m=1}^{n^k}\varepsilon_{m}x^{m}.
\end{equation}
We shall show that these sequences also satisfy the root-$N$ property
of Eq.~\eqref{equation:1a}. It turns out that we will exploit the
unitarity of the Fourier matrix.

\begin{thm}\label{thm:4}
  The sequence of coefficicents\/
  $\bigl(\varepsilon^{}_{m}\bigr)_{m\in\mathbb{N}}$ defined by
  Eq.~\eqref{eq:Pcomplex} together with the recurrence relation of
  Eq.~\eqref{equation:4a} satisfies the root-$N$ property of
  Eq.~\eqref{equation:1a}.
\end{thm}
\begin{proof}
  The proof proceeds by induction.  We want to derive a bound for
  $|P_{k+1}^{(1)}(x)|$. To do so, we express the sum of the squared
  norms of the polynomials $P_{k+1}^{(1)}(x),\ldots, P_{k+1}^{(n)}(x)$
  as
\[
  v_{k+1}^{\dagger}v^{}_{k+1} \, =\, 
  \sum_{j=1}^{n} \big\lvert P_{k+1}^{(j)}(x)\big\rvert^{2}.
\]
Using the recurrence relation and the identity $(A^{(n,k)})^{\dagger}A^{(n,k)}=n\, \boldsymbol{1}^{(n)}$, we obtain
\[
  v_{k+1}^{\dagger}v^{}_{k+1} \, =\, 
  v_{k}^{\dagger}(A^{(n,k)})^{\dagger}A^{(n,k)}v^{}_{k} \, =\, 
  n v_{k}^{\dagger}v^{}_{k} \, =\, 
  n\biggl(\sum_{j=1}^{n}\big\lvert P_{k}^{(j)}(x)\big\rvert^2\biggr).
\]
This shows that 
\[
  \sum_{j=1}^{n}\big\lvert P_{k+1}^{(j)}(x)\big\rvert^{2}=
  n\biggl(\sum_{j=1}^{n}\big\lvert P_{k}^{(j)}(x)\big\rvert^2\biggr).
\]
Since we have $\sum_{j=1}^{n}|P_{0}^{(j)}(x)|^{2}=n$ by the
initial conditions, we conclude by induction that
\[
  \sum_{j=1}^{n}\big\lvert 
   P_{k}^{(j)}(x)\big\rvert^{2}\, =\, n^{k+1}.
\]
Hence we get the bound
\begin{equation}\label{equation:4b}
  \big\lvert P_{k}^{(j)}(x)\big\rvert \, \leq\,
   n^{\frac{k+1}{2}},
\end{equation}
and in particular $\big\lvert P_{k}^{(1)}(x)\big\rvert \, \leq\,
n^{\frac{1}{2}}n^{\frac{k}{2}}$, which proves the root-$N$ property
for $N=n^{k}$.

It remains to prove the property for other values of $N$.  The
argument is similar to that used in the proof of
Proposition~\ref{prop:rootN}. Let $P_{k|m}^{(j)}$ denote
the $m$-th partial sum of $P_{k}^{(j)}$ for $1\le j\le n$, where
$n^{k-1}< m\leq n^{k}$. We will prove by induction that these functions
satisfy
 \begin{equation}\label{equation:4c}
   \big\lvert P_{k|m}^{(j)}(x)\big\rvert
   \,\leq\, G\,n^{\frac{k}{2}} 
\end{equation}
for all $|x|=1$ and $k\in\mathbb{N}_{0}$, where $G=n+n^\frac{1}{2}$.

Clearly, this estimate is true if $k=0$. Suppose now that
Eq.~\eqref{equation:4c} holds for some $k\in\mathbb{N}_0$, and
consider an integer $m$ with $n^k< m\leq n^{k+1}$. For $n^{k}<m\leq
2n^{k}$, by using the recursion \eqref{equation:4a} as well as the
triangle inequality together with Eqs.~\eqref{equation:4b} and
\eqref{equation:4c}, we obtain
\[
 \big\lvert P_{k+1|m}^{(j)}(x)\big\rvert  \, \leq\,
  \big\lvert P_{k}^{(1)}(x)\big\rvert +
  \big\lvert \omega^{j-1} x^{n^k} P_{k|m-n^{k}}^{(2)}(x)\big\rvert 
  \,\leq\,
  n^{\frac{k+1}{2}}+G\,n^{\frac{k}{2}} \, \leq\, G\,n^\frac{k+1}{2}
\]
for all $1\le j\le n$.

Similarly, we can derive bounds for the cases where $\ell n^{k}<m\leq
(\ell+1)n^{k}$ for all $1\le \ell\le n-1$, where more and more terms
contribute. We obtain
\begin{align*}
 \big\lvert P^{(j)}_{k+1|m}(x)\big\rvert \, &=\, 
\bigg\lvert \sum_{r=1}^{\ell} \omega^{(r-1)(j-1)}x^{(r-1) n^{k}}P_{k}^{(r)}(x) 
 \, + \, \omega^{\ell(j-1)}x^{\ell n^{k}}P_{k|m-\ell n^k}^{(\ell+1)}(x) \bigg\rvert\\
&\leq\,
\sum_{r=1}^{\ell} \big\lvert \omega^{(r-1)(j-1)}x^{(r-1) n^{k}}P_{k}^{(r)}(x)\big\rvert
 \, + \, \big\lvert\omega^{\ell(j-1)}x^{\ell n^{k}}
  P_{k|m-\ell n^k}^{(\ell+1)}(x) \big\rvert\\
&\leq\,
\sum_{r=1}^{\ell} \big\lvert P_{k}^{(r)}(x)\big\rvert
 \, + \, \big\lvert  P_{k|m-\ell n^k}^{(\ell+1)}(x) \big\rvert\\[2mm]
&\leq\, \ell\, n^{\frac{k+1}{2}} + G \, n ^{\frac{k}{2}} \\
&\leq\, ((n-1)+(n^\frac{1}{2}+1))\, n^{\frac{k+1}{2}} \, = \, G\, n^{\frac{k+1}{2}},
\end{align*}
which completes the induction argument.

To finish the proof, suppose that $n^{k-1} < N\leq n^{k}$. By
 Eq.~\eqref{equation:4c}, we have
 \[
  \bigl|P_{k|N}^{(1)}(x)\bigr|\,\leq\, 
  (n+n^{\frac{1}{2}})n^{\frac{k}{2}}\,\leq\,
  n(n^{\frac{1}{2}}+1)N^{\frac{1}{2}},
\]
which shows that the root-$N$ property holds.
\end{proof}
 
\begin{cor}\label{iic}
For any series of coefficients\/  $(\varepsilon_{n})^{}_{n\in\mathbb{N}}$ 
as in Theorem~\emph{\ref{thm:4}}, the corresponding
Dirac comb\/ $\sum_{n\in\mathbb{N}}\varepsilon_{n}\delta_{n}$ has
purely absolutely continuous diffraction.\qed
\end{cor}

Note that the case $n=2$ corresponds to Eq.~\eqref{equation:2a}, which
is the RS case. Let us now look at a couple of examples.

\begin{ex}
  Consider the case $n=3$. We start by setting
  $P_{0}(x)=Q_{0}(x)=R_{0}(x)=x$ and define polynomials $P_{k}$,
  $Q_{k}$ and $R_{k}$ recursively by
\begin{equation}\label{equation:4f}
\begin{pmatrix}
  P^{}_{k+1}(x)\\
  Q^{}_{k+1}(x)\\
  R^{}_{k+1}(x)
 \end{pmatrix} = 
 \begin{pmatrix}
  1 &   x^{3^k} & x^{2\cdot 3^{k}} \\
 1 &  \omega x^{3^{k}}& \omega^{2} x^{2\cdot 3^{k}}\\
 1 & \omega^{2} x^{3^{k}} & \omega  x^{2\cdot 3^{k}}
 \end{pmatrix}\begin{pmatrix}
  P^{}_{k}(x)\\
  Q^{}_{k}(x)\\
  R^{}_{k}(x)\end{pmatrix},
\end{equation}
where $k\in\mathbb{N}_{0}$ and $\omega=\exp(\frac{2\pi
  \mathrm{i}}{3})$.  From Theorem~\ref{thm:4}, we know that the
corresponding sequence of coefficients satisfies the root-$N$ property. This is now a ternary sequence in the alphabet
$\{1,\omega,\omega^{2}\}$.
 
As above, we can connect this to a substitution rule, where we now need 
nine letters. We denote these by $A$, $B$ and $C$ as well as 
the corresponding letters with a single and double bar. Here,
$A$, $B$ and  $C$ correspond to the coefficients of the polynomials
$P$, $Q$ and $R$, respectively, while the barred versions describe
the multiplication by $\omega$ (single bar) and $\omega^{2}$ (double bar).
Accordingly, we have $\overline{\overline{\overline{A}}}=A$ and similarly
for the other letters.

The structure of the matrix $A^{(3,k)}$ yields the following substitution rule
\[
  A\,\mapsto\, A\,B\,C, \qquad 
  B\,\mapsto\, A\,\overline{B}\,\overline{\overline{C}},
  \qquad  C\,\mapsto\, A\,\overline{\overline{B}}\,\overline{C},
\]
and the corresponding rules for the barred letters, leading to the
nine-letter substitution
 \[
  \begin{array}{lcr@{\qquad}lcr@{\qquad}lcr}
    A &\mapsto& A\,B\,C, &
    \overline{A}&\mapsto& \overline{A}\,\overline{B}\,\overline{C}, &
    \overline{\overline{A}}&\mapsto & \overline{\overline{A}}\,
     \overline{\overline{B}}\,\overline{\overline{C}},\\
    B&\mapsto & A\,\overline{B}\,\overline{\overline{C}}, &
    \overline{B}&\mapsto & \overline{A}\,\overline{\overline{B}}\,C, &
     \overline{\overline{B}}&\mapsto &\overline{\overline{A}}\,
     B\,\overline{C},\\
    C&\mapsto & A\,\overline{\overline{B}}\,\overline{C}, &
    \overline{C}&\mapsto & \overline{A}\,B\,\overline{\overline{C}}, &
     \overline{\overline{C}}&\mapsto & \overline{\overline{A}}\,
 \overline{B}\,C.
    \end{array}
\] 
This is a substitution of length $3$, which has a fixed point
obtained by iteration on the initial letter $A$,
\begin{align*}
  A\,& \mapsto\,  A\,B\,C \, \mapsto\, 
  A\,B\,C A\,\overline{B}\,\overline{\overline{C}}
  A\,\overline{\overline{B}}\,\overline{C}\\
  &\mapsto\,
  A\,B\,C 
  A\,\overline{B}\,\overline{\overline{C}}
  A\,\overline{\overline{B}}\,\overline{C}
  A\,B\,C 
  \overline{A}\,\overline{\overline{B}}\,C
  \overline{\overline{A}}\,\overline{B}\,C
  A\,B\,C 
  \overline{\overline{A}}\,B\,\overline{C}
  \overline{A}\,B\,\overline{\overline{C}}\,\mapsto\,\dotsb
\end{align*}
This fixed point is mapped to the ternary sequence of coefficients
of $P_{k}(x)$ by the factor map
\[
    \varphi^{(3)}\! : \; \begin{cases}
      A,B,C\mapsto 1,\\[0.5ex]
      \overline{A},\overline{B},\overline{C}\mapsto\omega,\\[0.5ex]
      \overline{\overline{A}},\overline{\overline{B}},
      \overline{\overline{C}}\mapsto\omega^2.\end{cases}  
\]
By Corollary~\ref{iic}, we know that the weighted Dirac comb
corresponding to this sequence has absolutely continuous diffraction.

Reasoning as before, we see that the three and nine-letter sequences
are in fact mutually locally derivable. Here it again suffices to
consider words of length four, many of which only have a single
pre-image under $\varphi^{(3)}$. An example is $111\omega^2$, for
which the only pre-image is $ABC\overline{\overline{A}}$. Due to
repetitiveness, we can therefore determine the three sublattices
locally, and hence locally derive the nine-letter sequence.
\end{ex}

\begin{ex}
  For our final example, we consider the case $n=4$.  Our recurrence
  relations are given by
\[
\begin{pmatrix}
  P^{}_{k+1}(x)\\
  Q^{}_{k+1}(x)\\
  R^{}_{k+1}(x)\\
  S^{}_{k+1}(x)
 \end{pmatrix} = 
 \begin{pmatrix}
  1 & x^{4^{k}} & x^{2\cdot 4^{k}} & x^{3\cdot 4^{k}} \\
   1 & \mathrm{i}x^{4^{k}} & -x^{2\cdot 4^{k}} & -\mathrm{i}x^{3\cdot 4^{k}}\\
   1 & -x^{4^{k}} & x^{2\cdot 4^{k}} & -x^{3\cdot 4^{k}}\\
  1 & -\mathrm{i}x^{4^{k}} & -x^{2\cdot 4^{k}} & +\mathrm{i}x^{3\cdot 4^{k}}
 \end{pmatrix}\begin{pmatrix}
   P^{}_{k}(x)\\
   Q^{}_{k}(x)\\
   R^{}_{k}(x)\\
   S^{}_{k}(x)\end{pmatrix}
\] 
with initial conditions $P^{}_{0}(x)=Q^{}_{0}(x)=R^{}_{0}(x)=S^{}_{0}(x)=x$.
In this case, we obtain the 16-letter substitution
\[
 A\, \mapsto\, A\,B\,C\,D,\qquad 
 B\, \mapsto\, A\,\overline{B}\,\overline{\overline{C}}\,
     \overline{\overline{\overline{D}}},\qquad
 C\, \mapsto\, A\,\overline{\overline{B}}\, C\,\overline{\overline{D}},\qquad
 D\, \mapsto\, A\,\overline{\overline{\overline{B}}}\,
      \overline{\overline{C}}\,\overline{D},
\]
together with the corresponding rules for the barred letters. The factor
map becomes
\[
    \varphi^{(4)}\! : \; \begin{cases}
      A,B,C,D\,\mapsto\, 1,\\[0.5ex]
      \overline{A},\overline{B},\overline{C},\overline{D}\, 
      \mapsto\, \mathrm{i},\\[0.5ex]
      \overline{\overline{A}},\overline{\overline{B}},
      \overline{\overline{C}},\overline{\overline{D}}\,\mapsto\, -1,\\[0.5ex]
      \overline{\overline{\overline{A}}},
      \overline{\overline{\overline{B}}},
      \overline{\overline{\overline{C}}},
      \overline{\overline{\overline{D}}}\,\mapsto\, -\mathrm{i}.
\end{cases}  
\]
As before, the five-letter and $25$-letter sequences are mutually
locally derivable, and again it is possible to find words of length
$4$ that only have a single ancestor under $\varphi^{(4)}$. One
example is $111\overline{1}$ (where $\overline{1}=-1$) whose ancestor
is $BCD\overline{\overline{A}}$.
\end{ex}

In the same way, starting from the $n\!\times\! n$ Fourier matrix, we
can construct substitution rules for any $n>1$, which all have
absolute continuous components in their spectra. The general structure
is clear from the examples above.  The substitutions act on $n^2$
letters, with $n$ `basic' letters that appear in $n$ different
`flavours' each (distinguished by the number of bars, from $0$ to
$n-1$). The distribution of bars in the image of the four basic
letters can be read off directly from the Fourier matrix, and the
remainder of the substitution is then fixed by cyclic symmetry under
the `bar' operation. The corresponding factor map $\varphi^{(n)}$
identifies all basic letters and the image only depends on the number
of bars, giving the corresponding power of $\exp(2\pi \mathrm{i}/n)$.

Theorem~\ref{thm:4} shows that the sequences of complex
numbers obtained by applying the factor map $\varphi^{(n)}$ satisfy
the root-$N$ property, and hence the corresponding Dirac comb has
purely absolutely continuous diffraction. We conjecture that the
dynamical spectrum of the $n^{2}$-letter hull of the substitution just
contains the absolute continuous component and the pure point
component corresponding to the maximum equicontinuous factor, which is
the corresponding solenoid. 

\section*{Acknowledgments}
The authors would like to thank Michael Baake and Franz
G\"{a}hler for many helpful comments as well as Fabien Durand and
Samuel Petite for extensive discussions on induced
systems.

\end{document}